\documentclass[12pt]{amsart}

\usepackage[a4paper,centering,margin=2.5cm]{geometry}
\usepackage{color}
\usepackage{amssymb}

\usepackage{enumerate}

\theoremstyle{plain}
\newtheorem{theorem}{Theorem}

\newtheorem{corollary}{Corollary}

\theoremstyle{definition}

\newtheorem{question}{Question}

\renewcommand{\le}{\leqslant}

\renewcommand{\ge}{\geqslant}

\makeatletter
\g@addto@macro{\endabstract}{\@setabstract}
\newcommand{\authorfootnotes}{\renewcommand\thefootnote{\@fnsymbol\c@footnote}}%
\makeatother

\begin{document}

\begin{center}
\Large 
On Ramsey numbers of complete graphs with dropped stars
\par\bigskip
\normalsize
\authorfootnotes
Jonathan Chappelon\footnote{Corresponding author}\footnote{E-mail address: jonathan.chappelon@um2.fr}\textsuperscript{1}, Luis Pedro Montejano\footnote{E-mail address: lpmontejano@gmail.com}\textsuperscript{1} and Jorge Luis Ram\'{i}rez Alfons\'{i}n\footnote{E-mail address: jramirez@um2.fr}\textsuperscript{1}
\par\bigskip
\textsuperscript{1}{Universit\'{e} Montpellier 2, Institut de Math\'{e}matiques et de Mod\'{e}lisation de Montpellier, Case Courrier 051, Place Eug\`{e}ne Bataillon, 34095 Montpellier Cedex 05, France}
\par\bigskip
October 18, 2014
\end{center}
\begin{abstract}
Let $r(G,H)$ be the smallest integer $N$ such that for any $2$-coloring (say, red and blue) of the edges of $K_n$, $n\ge N$, there is either a red copy of $G$ or a blue copy of $H$. Let $K_n-K_{1,s}$ be the complete graph on $n$ vertices from which the edges of  $K_{1,s}$ are dropped. In this note we present exact values for $r(K_m-K_{1,1},K_n-K_{1,s})$ and new upper bounds for $r(K_m,K_n-K_{1,s})$ in numerous cases. We also present some results for the Ramsey number of Wheels versus  $K_n-K_{1,s}$.\\[2ex]
\end{abstract}
\section{Introduction}
Let $G$ and $H$ be two graphs. Let $r(G,H)$ be the smallest integer $N$ such that for any $2$-coloring (say, red and blue) of the edges of $K_n$, $n\ge N$ there is either a red copy of $G$ or a blue copy of $H$. Let $K_n-K_{1,s}$ be the complete graph on $n$ vertices from which the edges of  $K_{1,s}$ are dropped. We notice that $K_n-K_{1,1}=K_n-e$ (the complete graph on $n$ vertices from which an edge is dropped) and $K_n-K_{1,2}=K_n-P_3$ (the complete graph on $n$ vertices from which a path on three vertices is dropped). 
\medskip

In this note we investigate $r(K_m-e,K_n-K_{1,s})$ and $r(K_m,K_n-K_{1,s})$ for a variety of integers $m,n$ and $s$. In the next section, we prove our main result (Theorem \ref{mainth}). In Section \ref{sec:1}, we will present exact values for $r(K_m-e,K_n-K_{1,s})$ when $n=3$ or $4$ and some values of $m$ and $s$. In Section \ref{sec:2}, new upper bounds for $r(K_m,K_n-P_3)$ for several integers $m$ and $n$ are given. In Section \ref{sec:3},  we give new upper bounds for $r(K_m,K_n-K_{1,s})$ when $m,s\ge 3$ and several values of $n$. In Section \ref{sec:4}, we present some equalities for $r(K_4,K_n-K_{1,s})$ extending the validity of some results given in \cite{BE89}. Finally, in Section \ref{sec:7}, we will present results concerning the Ramsey number of the Wheel $W_5$ versus  $K_n-K_{1,s}$. We present exact values for $r(W_5,K_6-K_{1,s})$ when $s=3$ and $4$ and the equalities $r(W_5,K_n-K_{1,s})=r(W_5,K_{n-1})$ when $n=7$ and $8$ for some values of $s$.

Some known  values/bounds for specific $r(K_m,K_n)$ needed for this paper are given in the Appendix.
\section{Main result}
Let $G$ be a graph and denote by $G^v$ the graph obtained from $G$ to which a new vertex $v$, incident to all the vertices of $G$, is added. Our main result is the following

\begin{theorem}\label{mainth}
Let $n$ and $s$ be positive integers. Let $G_1$ be any graph and let $N$ be an integer such that $N\ge r(G_1^v,K_n)$. If $\left\lceil\frac{(s+1)(N-n)}{n}\right\rceil\ge r(G_1,K_{n+1}-K_{1,s})$ then $r(G_1^v,K_{n+1}-K_{1,s}) \le N$.
\end{theorem}

\begin{proof}
Let $K_N$ be a complete graph on $N$ vertices and consider any 2-coloring of the edges of $K_N$ (say, red and blue). We shall show that there is either a $G_1^v$ red or a $K_{n+1}-K_{1,s}$ blue. Since $N\ge r(G_1^v,K_n)$ then $K_N$ has a red $G_1^v$ or a blue $K_n$. In the former case  we are done, so let us suppose that $K_N$ admit a blue $K_n$, that we will denote by $H$.  We have two cases.
\medskip

Case 1) There exists a vertex $u\in V(K_N\setminus H)$ such that $|N_H^r(u)|\le s$ where $N_H^r(u)$ is the set of vertices in $H$ that are joined to $u$ by a red edge. In this case, we may construct the blue graph $G'=K_{n+1}-K_{1,|N_H^r(u)|}$, this is done by taking $H$ (containing $n$ vertices) and vertex $u$ together with the blue edges between $u$ and the vertices of $H$. Now, since  $|N_H^r(u)|\le s$ then the graph $K_{n+1}-K_{1,s}$ is contained in $G'$ (and thus we found a blue $K_{n+1}-K_{1,s}$).
\medskip

Case 2) $|N_H^r(u)|> s$ for every vertex $u\in V(K_N\setminus H)$. Then we have that the number of red edges $\{x,y\}$ with $x\in V(H)$ and $y\in V(K_N\setminus H)$ is at least $(N-n)(s+1)$. So, by the pigeon hole principle, we have that there exists at least one vertex $v\in V(H)$ such that $d_{K_N\setminus H}^r(v)\ge \left\lceil\frac{(s+1)(N-n)}{n}\right\rceil$, where $d_{K_N\setminus H}^r(v)=\left|N_{K_N\setminus H}^r(v)\right|$ and  $N_{K_N\setminus H}^r(v)$ denotes the set of vertices in $K_N\setminus H$ incident to $v$ with a red edge. But since $\left\lceil\frac{(s+1)(N-n)}{n}\right\rceil\ge r(G_1,K_{n+1}-K_{1,s})$ then the graph induced by $N_{K_N\setminus H}^r(v)$ has either a blue $K_{n+1}-K_{1,s}$ (and we are done) or a red $G_1$ to which we add vertex $v$ to find a red $G^v$ as desired.
\end{proof}
\section{Some exact values for $r(K_m-e,K_n-K_{1,s})$}\label{sec:1}
Let $s\ge 1$ be an integer. We clearly have that
$$
r(K_3-e,K_m)\le r(K_3-e,K_{m+1}-K_{1,s}).
$$ 
Since 
$$
r(K_3-e,K_{m+1}-K_{1,s})\le r(K_3-e,K_{m+1}-e)
$$
and  (see \cite{Rad})
$$
r(K_3-e,K_{m})= r(K_3-e,K_{m+1}-e)=2m-1
$$ 
then
$$
\hbox{$r(K_3-e,K_{m+1}-K_{1,s})=2m-1$ for each $s=1,\dots, m-1$.}
$$  
\subsection{Case $m=4$.}
\begin{corollary}\label{exactvalues}
\begin{enumerate}[(a)]
\item[]
\item\label{C1a}
$r(K_4-e,K_5-K_{1,3})=11$.
\item\label{C1b}
$r(K_4-e,K_6-K_{1,s})=16$ for any $3\le s\le 4$.
\item\label{C1c}
$r(K_4-e,K_7-K_{1,s})=21$ for any $4\le s\le 5$.
\end{enumerate}
\end{corollary}

\begin{proof}
(\ref{C1a}) It is clear that $r(K_4-e,K_4)\le r(K_4-e,K_5-K_{1,3})$.  Since  $r(K_4-e,K_4)=11$ (see \cite{Rad}) then $11\le r(K_4-e,K_5-K_{1,3})$. We will now show that $r(K_4-e,K_5-K_{1,3})\le 11$.  By taking $N=11$, $s=3$ and $n=4$, we have that $\left\lceil\frac{(s+1)(N-n)}{n}\right\rceil=\left\lceil\frac{4\times 7}{4}\right\rceil=7 = r(K_3-e,K_5-K_{1,3})$ and so, by Theorem~\ref{mainth}, we have $r(K_4-e,K_5-K_{1,3})\le 11$, and the result follows.
\smallskip

The proofs for (\ref{C1b}) and (\ref{C1c}) are analogues. We just need to check that conditions of Theorem~\ref{mainth} are satisfied by taking : $N=r(K_4-e,K_5)=16$ for (\ref{C1b}) and $N=r(K_4-e,K_6)=21$  for (\ref{C1c}).
\end{proof}

We notice that Corollary~\ref{exactvalues}(\ref{C1a}) is claimed in \cite{Hen} without a proof. Corollary~\ref{exactvalues}(\ref{C1b}) can also be obtained by using that $r(K_4-e,K_6-P_3)=16$ \cite{Hoeth} since $16=r(K_4-e,K_6-P_3)\ge r(K_4-e,K_6-K_{1,s})\ge r(K_4-e,K_5) = 16$ for $s\in\{3,4\}$. Corollary~\ref{exactvalues}(\ref{C1c}) was first posed by Hoeth and Mengersen \cite{Hoeth}. The best known upper bounds for $r(K_4-e,K_7-K_{1,3})$ and $r(K_4-e,K_7-P_3)$  are obtained by applying the following classical recursive formula : 
\begin{equation}\label{recursive2}
r(K_m-e,K_n-K_{1,s})\le r(K_{m-1}-e,K_n-K_{1,s})+r(K_m-e,K_{n-1}-K_{1,s}).
\end{equation}

Hence 
$$
r(K_4-e,K_7-K_{1,3})\le r(K_3-e,K_7-K_{1,3})+ r(K_4-e,K_6-K_{1,3})=11+16=27
$$
and 
$$
r(K_4-e,K_7-P_3) \le r(K_3-e,K_7-P_3)+ r(K_4-e,K_6-P_3) = 11+16 = 27.
$$ 
We are able to improve the above upper bounds.

\begin{corollary}\label{cort3}
$21\le r(K_4-e,K_7-K_{1,3})\le 22$.
\end{corollary}

\begin{proof}
It is clear that $r(K_4-e,K_6)\le r(K_4-e,K_7-K_{1,3})$.  Since  $r(K_4-e,K_6)=21$ (see \cite{Rad}), then $21\le r(K_4-e,K_7-K_{1,3})$. We will now show that $r(K_4-e,K_7-K_{1,3})\le 22$.  By taking $N=22$, $s=3$ and $n=6$, we have that $\left\lceil\frac{(s+1)(N-n)}{n}\right\rceil=\left\lceil\frac{4\times 16}{6}\right\rceil=11=
r(K_3-e,K_7-K_{1,3})$ and so, by Theorem \ref{mainth}, we have that $r(K_4-e,K_7-K_{1,3})\le 22$, and the result follows.
\end{proof}

The above upper bound improves the previously best known one, given by $r(K_4-e,K_7-K_{1,3})\le 27$.

\subsection{Case $m=5$.}
The following equality is claimed in \cite{Hen} without a proof.

\begin{corollary}
$r(K_5-e,K_5-K_{1,3})=19$.
\end{corollary}

\begin{proof}
It is clear that $r(K_5-e,K_4)\le r(K_5-e,K_5-K_{1,3})$.  It is known that $r(K_5-e,K_4)=19$ (see \cite{Rad}), then $19\le r(K_5-e,K_5-K_{1,3})$. We will now show that $r(K_5-e,K_5-K_{1,3})\le 19$. By Corollary~\ref{exactvalues}, we have that $r(K_4-e,K_5-K_{1,3})=11$. Then, by taking $N=19$, $s=3$ and $n=4$, we have that$\left\lceil\frac{(s+1)(N-n)}{n}\right\rceil=\left\lceil\frac{4\times 15}{4}\right\rceil=15>r(K_4-e,K_5-K_{1,3})=11$ and so, by Theorem~\ref{mainth}, we have $r(K_5-e,K_5-K_{1,3})\le 19$, and the result follows.
\end{proof}

\begin{corollary}
$r(K_5-e,K_6-K_{1,s})=r(K_5-e,K_5)$ for $s=3,4$.
\end{corollary}

\begin{proof}
It is clear that $r(K_5-e,K_5)\le r(K_5-e,K_6-K_{1,s})$ for all $s\ge 1$. Let us now prove that $r(K_5-e,K_5)\ge r(K_5-e,K_6-K_{1,s})$ for $s=3,4$. Since $r(K_5-e,K_6-K_{1,4})\le r(K_5-e,K_6-K_{1,3})$ then it is sufficient to prove that $r(K_5-e,K_6-K_{1,3})\le r(K_5-e,K_5)$. For, let $N=r(K_5-e,K_5)\ge 30$ (this lower bound was proved by Exoo \cite{Exoo}). Since $N\ge 30$ then if $s=3$ and $n=5$ we obtain that $\left\lceil\frac{(s+1)(N-n)}{n}\right\rceil\ge\left\lceil\frac{4\times 25}{5}\right\rceil=20>17\ge r(K_4-e,K_6-K_{1,3})$ (see \cite{Rad} or Corollary~\ref{exactvalues}(\ref{C1b}) for the last inequality). So, by Theorem~\ref{mainth}, we obtain that $r(K_5-e,K_6-K_{1,3})\le N=r(K_5-e,K_5)$.
\end{proof}

We notice that in the case $s=2$, if $r(K_5-e,K_5)\ge 32$ then we may obtain that $r(K_5-e,K_6-K_{1,2})=r(K_5-e,K_5)$ (by using the same arguments as above). It is known that $r(K_5-e,K_5)\ge 30$.

\section{New upper bounds for $r(K_m,K_n-P_3)$}\label{sec:2}

In this section we will apply our main result to give new upper bounds for $r(K_m,K_n-P_3)$ in numerous cases. The value of $r(K_n,K_m-P_3)$ have already been studied in some cases. In \cite{B11,Calvert}, it is proved that $r(K_5,K_5-P_3)=25$ and in \cite{Clan} it is shown that $r(K_4,K_5-P_3)=r(K_4,K_4)=18$.
\medskip

Let us first notice that, by taking $G_1=K_m$ in Theorem~\ref{mainth}, we obtain

\begin{corollary}\label{mainc}
Let $N$ be an integer such that $N\ge r(K_{m+1},K_n)$. If $\left\lceil\frac{(s+1)(N-n)}{n}\right\rceil\ge r(K_m,K_{n+1}-K_{1,s})$ then $r(K_{m+1},K_{n+1}-K_{1,s})\le N$.
\end{corollary}

The case when $m=3$ has already been studied in \cite{BBH98} where it is proved that 
$$
\hbox{$r(K_3,K_{n+1}-K_{1,s})=r(K_3,K_n)$ if $n\ge s+1>(n-1)(n-2)/(r(3,n)-n)$.}
$$
As a consequence,  we have
\begin{equation}\label{eq1}
\begin{array}{ll}
r(K_3,K_{6}-P_3)=r(K_3,K_5) & (\text{with } n=5 \text{ and } s=2),\\
r(K_3,K_{7}-K_{1,3})=r(K_3,K_6) & (\text{with } n=6 \text{ and } s=3),\\
r(K_3,K_{10}-K_{1,s})=r(K_3,K_9) & (\text{with } n=9 \text{ for any } 2\le s\le 9),\\
r(K_3,K_{11}-K_{1,s})=r(K_3,K_{10}) & (\text{with } n=10 \text{ for any } 3\le s\le 10).\\
\end{array}
\end{equation}

\subsection{Results on $r(K_m,K_5-P_3)$}

In \cite[Theorem 4]{BE89}, it was shown that if $n\ge m\ge 3$ and $m+n\ge 8$, then
\begin{equation}\label{eqer}
\hbox{$r(K_{m+1}-K_{1,m-p},K_{n+1}-K_{1,n-q})=r(K_m,K_n)$ where $p=\left\lceil\frac{m}{n-1}\right\rceil$ and $q=\left\lceil\frac{n}{m-1}\right\rceil$.}
\end{equation}

This result implies the following 

\begin{corollary}\label{erdos}
Let $n\ge m\ge 3$ and $m+n\ge 8$ and let $p=\left\lceil\frac{m}{n-1}\right\rceil$ and $q=\left\lceil\frac{n}{m-1}\right\rceil$. Then,
$$
r(K_{m},K_{n+1}-K_{1,n-q})=r(K_{m+1}-K_{1,m-p},K_n)=r(K_m,K_n).
$$
\end{corollary}

\begin{proof} We clearly have
$$
r(K_m,K_n)\le r(K_{m},K_{n+1}-K_{1,n-q})\le r(K_{m+1}-K_{1,m-p},K_{n+1}-K_{1,n-q})\stackrel{\eqref{eqer}}{=}r(K_m,K_n)
$$
and thus $r(K_{m},K_{n+1}-K_{1,n-q})=r(K_m,K_n)$ (the proof for $r(K_{m+1}-K_{1,m-p},K_n)=r(K_m,K_n)$ is similar).
\end{proof}

By taking $m=n=4$ (and thus $q=2$) in Corollary~\ref{erdos} we have that
$$
r(K_4,K_5-P_3)=r(K_4,K_4)=18.
$$
It is also known \cite{B11} that 
$$
r(K_5,K_5-P_3)=r(K_5,K_4)=25,
$$ 
and, by Corollary~\ref{erdos}, we have
\begin{equation}\label{eq1a}
\begin{array}{ll}
r(K_6,K_{4}-P_3)=r(K_6,K_3)=18 & (\text{with } m=5 \text{ and } n=3),\\
r(K_7,K_{4}-P_3)=r(K_7,K_3)=23 & (\text{with } m=6 \text{ and } n=3),\\
r(K_8,K_{4}-P_3)=r(K_8,K_3)=28 & (\text{with } m=7 \text{ and } n=3),\\
r(K_9,K_{4}-P_3)=r(K_9,K_3)=36 & (\text{with } m=8 \text{ and } n=3),\\
r(K_{10},K_{4}-P_3)=r(K_{10},K_3)\le 43 & (\text{with } m=9 \text{ and } n=3).\\
\end{array}
\end{equation}

The best known upper bounds of $r(K_n,K_5-P_3)$ for $n\ge 6$ are obtained by applying the following classical recursive formula : 
\begin{equation}\label{recursive}
r(K_m,K_n-K_{1,s})\le r(K_{m-1},K_n-K_{1,s})+r(K_m,K_{n-1}-K_{1,s}).
\end{equation}

By using \eqref{eq1a}, we obtain
$$
\begin{array}{l}
r(K_6,K_5-P_3)\le r(K_5,K_5-P_3)+r(K_6,K_4-P_3)=25+r(K_6,K_3)=25+18=43, \\
r(K_7,K_5-P_3)\le r(K_6,K_5-P_3)+r(K_7,K_4-P_3)= 43+23=66, \\
r(K_8,K_5-P_3) \begin{array}[t]{l} \le r(K_7,K_5-P_3)+r(K_8,K_4-P_3) \\
\le r(K_6,K_5-P_3)+r(K_7,K_4-P_3)+28=43+23+28=94,
\end{array} \\
r(K_9,K_5-P_3)\le r(K_8,K_5-P_3)+r(K_9,K_4-P_3)=94+36=130, \\
r(K_{10},K_5-P_3) \begin{array}[t]{l} \le r(K_9,K_5-P_3)+r(K_{10},K_4-P_3) \\
\le r(K_8,K_5-P_3)+r(K_9,K_4-P_3)+43=94+36+43=173. \\
\end{array}
\end{array}
$$
We are able to improve all the above upper bounds.

\begin{corollary}\label{cor1}
\begin{enumerate}[(a)]
\item[]
\item\label{C7a}
$r(K_6,K_5-P_3)\le 41$.
\item\label{C7b}
$r(K_7,K_5-P_3)\le 61$.
\item\label{C7c}
$r(K_8,K_{5}-P_3)\le 85$.
\item\label{C7d}
$r(K_9,K_{5}-P_3)\le 117$.
\item\label{C7e}
$r(K_{10},K_{5}-P_3)\le 159$.
\end{enumerate}
\end{corollary}

\begin{proof}
(\ref{C7a}) It is known that  $r(K_6,K_4)\le 41$. Then, by taking $N=41$, $s=2$ and $n=4$, we have that $\left\lceil\frac{(s+1)(N-n)}{n}\right\rceil=\left\lceil\frac{3\times 37}{4}\right\rceil=28 > r(K_5,K_5-P_3)=25$ and so, by Corollary~\ref{mainc}, the result follows.
\smallskip

The proofs for the rest of the cases are analogues. We just need to check that conditions are satisfied  by taking: $N=61\ge r(K_7,K_4)$ for (\ref{C7b}), $N=85>84\ge r(K_8,K_4)$ for (\ref{C7c}), $N=117>115\ge r(K_9,K_4)$ for (\ref{C7d}) and $N=159>149\ge r(K_{10},K_4)$ for (\ref{C7e}).
\end{proof}

By applying recursion \eqref{recursive} to $r(K_{11},K_{5}-P_3)$ one may obtain that $r(K_{11},K_{5}-P_3)\le 224$ if the old known values are used in the recursion, and it can be improved to $r(K_{11},K_{5}-P_3)\le 210$ by using the new values given in Corollary~\ref{cor1}. The latter beats the upper bound $r(K_{11},K_{5}-P_3)\le 215$ obtained via Corollary~\ref{mainc}.
\medskip

We can also use Corollary~\ref{mainc} to give the following equality.

\begin{corollary}
If $37\le r(K_6,K_4)$ then $r(K_6,K_5-P_3)=r(K_6,K_4)$.
\end{corollary}

\begin{proof}
It is clear that $r(K_6,K_4)\le r(K_6,K_5-P_3)$. We show that $r(K_6,K_5-P_3)\le r(K_6,K_4)$. Let $N=r(K_6,K_4)\ge 37$. Since $N\ge 37$ and by taking $s=2$ and $n=4$ we have
$\left\lceil\frac{(s+1)(N-n)}{n}\right\rceil\ge \left\lceil\frac{3\times 33}{4}\right\rceil=25=r(K_5,K_5-P_3)$, and so, by Corollary~\ref{mainc}, $r(K_6,K_5-P_3)\le N=r(K_6,K_4)$. 
\end{proof}

It is known that $36\le r(K_6,K_4)$. In the case when $r(K_6,K_4)=36$ the above result might not hold. 

\subsection{Results on $r(K_m,K_6-P_3)$}

Since $r(K_3,K_5)=14$ then, by \eqref{eq1} we have $r(K_3,K_6-P_3)=14$ \cite{Faudree}. So, by \eqref{recursive}, we have
$$
r(K_4,K_6-P_3)\le r(K_3,K_6-P_3)+r(K_4,K_5-P_3)=14+18=32.
$$ 
Moreover, it is known that the upper bound is strict if the terms of the right side are even, which is our case, and so, $r(K_4,K_6-P_3)\le 31$.

\begin{corollary}\label{ccc1}
\begin{enumerate}[(a)]
\item[]
\item\label{C9a}
$25\le r(K_4,K_6-P_3)\le 27$.
\item\label{C9b}
$r(K_5,K_6-P_3)\le 49$.
\item\label{C9c}
$r(K_6,K_6-P_3)\le 87$.
\end{enumerate}
\end{corollary}

\begin{proof}
(\ref{C9a}) We clearly have that $25=r(K_4,K_5)\le r(K_4,K_6-P_3)$. It is known that $r(K_4,K_5)=25$. We take $N=27>r(K_4,K_5)$, $s=2$ and $n=5$. So, $\left\lceil\frac{(s+1)(N-n)}{n}\right\rceil=\left\lceil\frac{3\times 22}{5}\right\rceil=14=r(K_3,K_6-P_3)$ and so, by Corollary~\ref{mainc}, $r(K_4,K_6-P_3)\le 27$.
\smallskip

The proofs for (\ref{C9b}) and (\ref{C9c}) are analogues. We just need to check that conditions of Corollary~\ref{mainc} are satisfied  by taking: $N=49\ge r(K_5,K_5)$ for (\ref{C9b}) and $N=87\ge r(K_6,K_5)$ for (\ref{C9c}).
\end{proof}

The recursive formula~\eqref{recursive} gives now (by using the new above values) $r(K_{7},K_{6}-P_3)\le 148$ (before, by using the old values, it gave $158$). This new upper bound beats the upper bound $r(K_{7},K_{6}-P_3)\le 149$ obtained by Corollary~\ref{mainc}.

\subsection{Results on $r(K_m,K_n-P_3)$ for a variety of $m$ and $n$}

\begin{corollary}
For each $3\le m\le 5$ and each $7\le n\le 16$, we have that $r(K_m,K_n-P_3)\le u(m,n)$, where the value of $u(m,n)$ is given in the $(m,n)$ entry of the below table (the value between parentheses is the best previously known upper bound).
\begin{center}
\tiny
\noindent\makebox[\textwidth]{
\begin{tabular}{|c|c|c|c|c|c|c|c|c|c|c|}
\hline
 $m\setminus n$ & $7$ & $8$ & $9$ & $10$ & $11$ & $12$ & $13$ & $14$ & $15$ & $16$ \\
\hline
 $3$ & & & & & $44(47)$ & $52(59)$ & $61(72)$ & $70(86)$ & $80(101)$ & $91(117)$ \\
\hline
 $4$ & $41(49)$ & $61(72)$ & & $115(136)$ & $154(183)$ & $199(242)$ & $253(319)$ & $313(405)$ & $383(506)$ & $466(623)$ \\
\hline
 $5$ & $87(105)$ & $143(177)$ & $222(277)$ & & & & & &  & \\
\hline
\end{tabular}}
\end{center}
\end{corollary}

\begin{proof}
We just need to check that conditions of Corollary~\ref{mainc} are satisfied  by taking: $N=41\ge r(K_4,K_6)$ for $u(4,7)$, $N=87\ge r(K_5,K_6)$ for $u(5,7)$, $N=61\ge r(K_4,K_7)$ for $u(4,8)$, $N=143\ge r(K_5,K_7)$ for $u(5,8)$, $N=222>216\ge r(K_5,K_8)$ for $u(5,9)$, $N=115\ge r(K_4,K_9)$ for $u(4,10)$, $N=47>42\ge r(K_3,K_{10})$ for $u(3,11)$, $N=154>149\ge r(K_4,K_{10})$ for $u(4,11)$, $N=52>51\ge r(K_3,K_{11})$ for $u(3,12)$, $N=199>191\ge r(K_4,K_{11})$ for $u(4,12)$, $N=61>59\ge r(K_3,K_{12})$ for $u(3,13)$, $N=253>238\ge r(K_4,K_{12})$ for $u(4,13)$, $N=70>69\ge r(K_3,K_{13})$ for $u(3,14)$, $N=313>291\ge r(K_4,K_{13})$ for $u(4,14)$, $N=80>78\ge r(K_3,K_{14})$ for $u(3,15)$, $N=383>349\ge r(K_4,K_{14})$ for $u(4,15)$,  $N=91>88\ge r(K_3,K_{15})$ for $u(3,16)$, $N=466>417\ge r(K_4,K_{15})$ for $u(4,16)$.
\end{proof}

\section{Some bounds for $r(K_m,K_n-K_{1,s})$ when $s\ge 3$}\label{sec:3}
Here, we will focus our attention to  upper bounds for $r(K_m,K_n-K_{1,3})$ that yields to upper bounds for $r(K_m,K_n-K_{1,s})$ when $s\ge 4$ since
$$
\hbox{$r(K_m,K_n-K_{1,s})\le r(K_m,K_n-K_{1,3})$ for all $s\ge 4$.}
$$

\subsection{Results on $r(K_m,K_6-K_{1,3})$.}
In \cite{BE89} it was proved that $r(K_5,K_6-K_{1,3})=r(K_5,K_5)\le 49$. So by \eqref{recursive} we have
$$
r(K_6,K_6-K_{1,3})\le r(K_5,K_6-K_{1,3})+r(K_6,K_5-K_{1,3})=49+41=90.
$$

\begin{corollary}
For each $6\le m\le 15$, we have that $r(K_m,K_6-K_{1,3})\le u(m)$, where
the value of $u(m)$ is given in the below table (the value between parentheses is the best previously known upper bound).
\begin{center}
\tiny
\noindent\makebox[\textwidth]{
\begin{tabular}{|c|c|c|c|c|c|c|c|c|c|c|}
\hline
 $m$ & $6$ & $7$ & $8$ & $9$ & $10$ & $11$ & $12$ & $13$ & $14$ & $15$ \\
\hline
 $b_u$ & $87(90)$ & $143(151)$ & $216(235)$ & $316(350)$ & $442(499)$ & $633(690)$ & $848(928)$ & $1139(1219)$ & $1461(1568)$ & $1878(1568)$ \\
\hline
\end{tabular}}
\end{center}
\end{corollary}

\begin{proof}
It follows by Corollary~\ref{mainc} and by taking $N$ as the best known upper bound of $r(K_n,K_5)$ for each $n=6,\dots ,15$.
\end{proof}

We notice that by using similar arguments as above, we could prove that  $r(K_6,K_6-K_{1,3})=r(K_6,K_5)$ if $66\le r(K_6,K_5)$ .

\subsection{Results on $r(K_m,K_7-K_{1,3})$}
In \cite{BBH98} it was proved that $r(K_3,K_7-K_{1,3})=18$. Since $r(K_3,K_6)=18$ then, by \eqref{eq1} we have $r(K_3,K_7-K_{1,3})=18$. So, by \eqref{recursive}, we have
$$
r(K_4,K_7-K_{1,3})\le r(K_3,K_7-K_{1,3})+r(K_4,K_6-K_{1,3})=18+25=43.
$$

\begin{corollary}
For each $4\le m\le 11$, we have that $r(K_m,K_7-K_{1,3})\le u(m)$, where
the value of $u(m)$ is given in the below table (the value between parentheses is the best previously known upper bound).
\begin{center}
\noindent\makebox[\textwidth]{
\begin{tabular}{|c|c|c|c|c|c|c|c|c|}
\hline
 $m$ & $4$ & $5$ & $6$ & $7$ & $8$ & $9$ & $10$ & $11$ \\
\hline
 $b_u$ & $41(43)$ & $87(90)$ & $165(180)$ & $298(331)$ & $495(566)$ & $780(916)$ & $1175(1415)$ & $1804(2105)$ \\
\hline
\end{tabular}}
\end{center}
\end{corollary}

\begin{proof}
It follows by Corollary~\ref{mainc}, by taking $s=3$ and $N$ equals to the best known upper bound for $r(K_n,K_6)$ when $n=5,6,7,8,9,11$ and $N=1175>1171\ge r(K_{10},K_6)$ when $n=10$. For instance, for (1) we take $N=41\ge r(K_4,K_6)$, $s=3$ and $n=6$. Then,$\left\lceil\frac{(s+1)(N-n)}{n}\right\rceil=\left\lceil\frac{4\times 35}{6}\right\rceil=24>r(K_3,K_7-K_{1,3})$ and, by Corollary~\ref{mainc}, $r(K_4,K_7-K_{1,3})\le 41$.
\end{proof}

\section{More equalities}\label{sec:4}

From \eqref{eqer} we have  that $r(K_4,K_{n+1}-K_{1,s})=r(K_4,K_n)$ if $s\ge n-\left\lceil\frac{n}{3}\right\rceil$. The latter yields to the following equalities.
$$
\begin{array}{ll}
r(K_4,K_7-K_{1,s})=r(K_4,K_6) \text{ if } s\ge 4,
&
r(K_4,K_8-K_{1,s})=r(K_4,K_7) \text{ if } s\ge 5, \\
r(K_4,K_9-K_{1,s})=r(K_4,K_8) \text{ if } s\ge 5,
&
r(K_4,K_{10}-K_{1,s})=r(K_4,K_9) \text{ if } s\ge 6, \\
r(K_4,K_{11}-K_{1,s})=r(K_4,K_{10}) \text{ if } s\ge 6,
&
r(K_4,K_{12}-K_{1,s})=r(K_4,K_{11}) \text{ if } s\ge 7, \\
r(K_4,K_{13}-K_{1,s})=r(K_4,K_{12}) \text{ if } s\ge 8,
&
r(K_4,K_{14}-K_{1,s})=r(K_4,K_{13}) \text{ if } s\ge 8, \\
r(K_4,K_{15}-K_{1,s})=r(K_4,K_{14}) \text{ if } s\ge 9,
&
r(K_4,K_{16}-K_{1,s})=r(K_4,K_{15}) \text{ if } s\ge 10.\\
\end{array}
$$
We are able to extend all these equalities for further values of $s$.
\begin{corollary}\label{coo1}\ \\
\begin{tabular}{@{\hspace{1pt}}cc}
\begin{minipage}{0.49\textwidth}
\small
\begin{enumerate}[(a)]
\item\label{C13a}
$r(K_4,K_7-K_{1,s})=r(K_4,K_6)$ for $s=3$.
\stepcounter{enumi}
\item\label{C13c}
$r(K_4,K_9-K_{1,s})=r(K_4,K_8)$ for $s=4$.
\stepcounter{enumi}
\item\label{C13e}
$r(K_4,K_{11}-K_{1,s})=r(K_4,K_{10})$ for $s=5$.
\stepcounter{enumi}
\item\label{C13g}
$r(K_4,K_{13}-K_{1,s})=r(K_4,K_{12})$ for $s=6,7$.
\stepcounter{enumi}
\item\label{C13i}
$r(K_4,K_{15}-K_{1,s})=r(K_4,K_{14})$ for $s=8$.
\end{enumerate}
\end{minipage}
&
\begin{minipage}{0.49\textwidth}
\small
\begin{enumerate}[(a)]
\stepcounter{enumi}
\item\label{C13b}
$r(K_4,K_8-K_{1,s})=r(K_4,K_7)$ for $s=3,4$.
\stepcounter{enumi}
\item\label{C13d}
$r(K_4,K_{10}-K_{1,s})=r(K_4,K_9)$ for $s=4,5$.
\stepcounter{enumi}
\item\label{C13f}
$r(K_4,K_{12}-K_{1,s})=r(K_4,K_{11})$ for $s=6$.
\stepcounter{enumi}
\item\label{C13h}
$r(K_4,K_{14}-K_{1,s})=r(K_4,K_{13})$ for $s=7$.
\stepcounter{enumi}
\item\label{C13j}
$r(K_4,K_{16}-K_{1,s})=r(K_4,K_{15})$ for $s=9$.
\end{enumerate}
\end{minipage}
\end{tabular}
\end{corollary}

\begin{proof}
(\ref{C13a}) Since $r(K_4,K_6)\ge 36$ it follows that $r(K_4,K_7-K_{1,3})\ge 36$ and by \eqref{eq1}, we have $r(K_3,K_7-K_{1,3})=r(K_3,K_6)=18$. Let us take $N=r(K_4,K_6)\ge 36$, $s=3$ and $n=6$. So, $\left\lceil\frac{(s+1)(N-n)}{n}\right\rceil\ge\left\lceil\frac{4\times 30}{6}\right\rceil=20>r(K_3,K_7-K_{1,3})=18$ and the result follows by Corollary~\ref{mainc}.
\smallskip

The proofs for the rest of the cases are analogues. We just need to check that conditions of Corollary~\ref{coo1} are satisfied  by taking: $N=r(K_4,K_7)\ge 49$ and checking that $r(K_3,K_8-K_{1,3})=r(K_3,K_7)=23$ for (\ref{C13b}), $N=r(K_4,K_8)\ge 58$ and checking that $r(K_3,K_9-K_{1,4})=r(K_3,K_8)=28$ for (\ref{C13c}) and so on. 
\end{proof}

We notice that, by using the same arguments as above, we could improve cases (\ref{C13e}) and (\ref{C13g}) by showing that $r(K_4,K_{11}-K_{1,4})=r(K_4,K_{10})$ when $r(K_4,K_{10})\neq 92$ and $r(K_4,K_{13}-K_{1,5})=r(K_4,K_{12})$ when $r(K_4,K_{12})\neq 128$. 
\smallskip

In view of Corollary~\ref{coo1}, we may pose the following question,

\begin{question}
Let $n\ge 7$ be an integer. For which integer $s$ the equality $r(K_4,K_n-K_{1,s})=r(K_4,K_{n-1})$ holds?
\end{question}

Or more ambitious, in view of \cite[Theorem 4]{BE89}, we may pose the following,

\begin{question}
Let $m\ge 4$ and $n\ge 7$ be integers. For which integer $s\le n-1$ the equality $r(K_m,K_n-K_{1,s})=r(K_m,K_{n-1})$ holds?
\end{question}

\section{Wheels versus $K_n-K_{1,s}$}\label{sec:7}

In this section we obtain further relating results by applying Theorem~\ref{mainth} to other graphs. Indeed, we may consider $G_1$ as the cycle on $n-1$ vertices $C_{n-1}$, and thus  $G^v_1$ will be the wheel $W_n$ by taking the new vertex $v$ incident to all the vertices of $C_{n-1}$.  

\begin{corollary}
\begin{enumerate}[(a)]
\item[]
\item\label{C14a}
$r(W_5,K_6-K_{1,s})=27$ for $s=3,4,5$.
\item\label{C14b}
$r(W_5,K_7-K_{1,s})=r(W_5,K_6)$ for $s=4,5,6$.
\item\label{C14c}
$r(W_5,K_8-K_{1,s})=r(W_5,K_7)$ for $s=4,5,6,7$.
\end{enumerate}
\end{corollary}

\begin{proof}
(\ref{C14a}) It is clear that $r(W_5,K_5)\le r(W_5,K_6-K_{1,s})$ for any $1\le s\le 5$.  Since  $r(W_5,K_5)=27$ (see \cite{Rad}), then $27\le r(W_5,K_6-K_{1,s})$. We will now show that $r(W_5,K_6-K_{1,s})\le 27$ for $3\le s\le 5$.  By taking $N=27$, $s\ge 3$ and $n=5$, we have that $\left\lceil\frac{(s+1)(N-n)}{n}\right\rceil\ge\left\lceil\frac{4\times 22}{5}\right\rceil=18=
r(C_4,K_6)\ge r(C_4,K_6-K_{1,s})$ and so, by Theorem~\ref{mainth}, we have $r(W_5,K_6-K_{1,s})\le 27$, and the result follows.
\medskip

The proofs for (\ref{C14b}) and (\ref{C14c}) are analogues. We just need to check that conditions of Theorem~\ref{mainth} are satisfied by taking: $N=r(W_5,K_6)\ge 33$ for (\ref{C14b}) and $N=r(W_5,K_7)\ge 43$ for (\ref{C14c}) (see \cite{Rad} for the lower bounds of $r(W_5,K_6)$ and $r(W_5,K_7)$).
\end{proof}

\section*{Acknowledgments}
The authors would like to thank the anonymous referees for their useful remarks and comments.

\section*{Appendix}
The following table was obtained from \cite{Rad}.

\begin{center}
\begin{table}[h] 
\centering
\footnotesize
\noindent\makebox[\textwidth]{%
\begin{tabular}{|c|c|c|c|c|c|c|c|c|}
 \hline
 & $K_3$ & $K_4$ &$K_5$ &$K_6$ &$K_7$ &$K_8$ &$K_9$ &$K_{10}$\\
\hline
$K_3$ & 6 & 9 & 14 & 18 & 23 & 28 & 36 & [40,42]\\
\hline
$K_4$ & &18 & 25 & [36,41] & [49, 61] & [58,84] & 73,115] & [92,149]\\
\hline
$K_5$  & & & [43,49] & [58,87] & [80,143] & [101,216] & [126,316] & [144,442]\\
\hline
$K_6$  & & & & [102,165] & [113,298] & [132,495] & [169,780] & [179,1171]\\ 
\hline          
\end{tabular}}
\vspace{.2cm}
\caption{\label{current1}Some known bounds and values of $r(K_m,K_n)$.}
\end{table}
\end{center}

\end{document}